\documentclass[oneside,11pt]{amsart}
\usepackage{amssymb,amscd,amsmath,latexsym}
\usepackage[mathcal]{euscript}
\usepackage{hyperref}

\newtheorem{thm}{Theorem}[section]
\newtheorem{cor}[thm]{Corollary}
\newtheorem{lem}[thm]{Lemma}
\newtheorem{prop}[thm]{Proposition}
\theoremstyle{definition}
\newtheorem{defin}[thm]{Definition}

\newtheorem{conj}[thm]{Conjecture}
\newtheorem{prob}[thm]{Problem}
 \numberwithin{equation}{section}
 \numberwithin{figure}{section}

\theoremstyle{remark}
\newtheorem{rmk}[thm]{Remark}

\newcommand{\pbp}{$p  \times p$ }
\newcommand{\Res}{\operatorname{Res}}
\begin{document}

\title[Specht module complexity]
{\bf The complexity of certain Specht modules for the symmetric group}
\author{\sc David J. Hemmer}
\address{Department of Mathematics\\ University at Buffalo, SUNY \\
244 Mathematics Building\\Buffalo, NY~14260, USA}
\thanks{Research of the  author was supported in part by NSF
grant  DMS-0556260} \email{dhemmer@math.buffalo.edu}

\date{November 2008}


\begin{abstract}
 During the 2004-2005 academic year the VIGRE algebra research group at the University of Georgia computed the complexities of certain Specht modules $S^\lambda$ for the symmetric group $\Sigma_d$, using the computer algebra program Magma. The complexity of an indecomposable module does not exceed the $p$-rank of the defect group of its block. The Georgia group conjectured that, generically, the complexity of a Specht module attains this maximal value; that it is smaller precisely when the Young diagram of $\lambda$ is built out of $p \times p$ blocks. We prove one direction of this conjecture by showing these Specht modules do indeed have less than maximal complexity. It turns out that this class of partitions, which has not previously appeared in the literature, arises naturally as the solution to a question about the $p$-weight of partitions and branching.
\end{abstract}
\maketitle

\section{Introduction}
\label{sec: Introduction}
\subsection{}
In 2004 the VIGRE algebra research group at the University of Georgia computed some examples of the complexity of Specht modules $S^\lambda$ for the symmetric group $\Sigma_d$. Their data led them to focus on partitions $\lambda$ of a curious form, specifically:

\begin{defin}
\label{defin: pxpdef} A partition $\lambda \vdash d$ is $p \times p$  if $\lambda=(\lambda_1^{a_1}, \lambda_2^{a_2}, \ldots, \lambda_s^{a_s})$ where $p\mid \lambda_i$ and $p \mid a_i$ for all $i$.
\end{defin}
Such $\lambda$ can exist only if $p^2 \mid d$. Equivalently, $\lambda$ is \pbp if both $\lambda$ and its transpose $\lambda'$ are of the form $p\tau$. Also equivalently, the Young diagram of $\lambda$ is built from $p \times p$ blocks.

Now suppose $S^\lambda$ is in a block $B(\lambda)$ of weight $w$ corresponding to a $p$-core $\tilde{\lambda} \vdash d-pw$. Then the defect group of $B(\lambda)$ is isomorphic to a Sylow $p$-subgroup of $\Sigma_{pw}$ and has $p$-rank $w$. In particular, the maximum complexity of any module in the block $B(\lambda)$  is $w$. The VIGRE group made the following conjecture:

\begin{conj}[UGA VIGRE \footnote{This conjecture and some discussion can be found at \href{http://www.math.uga.edu/~nakano/vigre/vigre.html}
{http://www.math.uga.edu/$\sim$ nakano/vigre/vigre.html}.}]
\label{conj: Vigre}
Let $S^\lambda$ be in a block $B$ of weight $w$. Then the complexity of $S^\lambda$ is $w$ if and only if $\lambda$ is not \pbp.
\end{conj}
Conjecture \ref{conj: Vigre} implies that almost every Specht module has maximal complexity among modules in its block. Indeed it would imply that if $p^2\not\,\mid d$, then all the Specht modules for $\Sigma_d$ have this property. As far as we know the condition we call \pbp has not appeared anywhere in the literature, and it seems quite mysterious. However we will prove that it arises quite naturally from considering the weights of $\Sigma_d$ blocks and the branching theorems. Specifically we prove:

\begin{thm}
\label{thm: main}
Suppose $\lambda \vdash d$ has $p$-weight $w$. Then $\lambda$ is \pbp if and only if $w(\lambda_A) \leq w-2$ for each removable node $A$ of $\lambda$. In this case, $w(\lambda_A)$ is always equal to $w-2$.
\end{thm}
Theorem \ref{thm: main} give one direction of Conjecture \ref{conj: Vigre} as a fairly immediate corollary.

\begin{cor}
\label{cor: onewayofconjecture}
Suppose $\lambda \vdash p^2d$ is \pbp, and hence of weight $w=pd$. Then the complexity of $S^\lambda$ is less than $w$.
\end{cor}

\begin{rmk}
\label{rmk: d<p^2}When $d<p^2$ Conjecture \ref{conj: Vigre} can be deduced immediately from the dimensions of the corresponding Specht modules, as was noticed by the VIGRE group and proven by Lim \cite[Thm. 4.1]{LimVarietySpechtpreprint}.
\end{rmk}

\begin{rmk}
\label{rmk: Lim}
The conjecture was verified for $d=p^2$ by Lim in \cite[Thm. 3.1]{LimVarietySpechtpreprint}, where he showed  $S^{(p^p)}$ has complexity $p-1$ even though $(p^p)$ has weight $p$. Lim's proof uses the fact that $\Res_{\Sigma_{p^2-1}}(S^{(p^p)}) \cong S^{(p^{p-1},p-1)}$ and $(p^{p-1},p-1)$ has weight $p-2$. Our proof is essentially a large generalization of this observation about the branching theorems.

\end{rmk}

\section{Blocks and abaci}
\subsection{}
In this section we review the description of the blocks of $k\Sigma_d$ and the representation of partitions of $d$ on the abacus. For further details on this, and as good sources for the representation theory of the symmetric group, see the books \cite{JamesKerberbook} and \cite{Jamesbook}, where definitions of basic terms not defined here, like removable node, addable node, residue of a node, etc.., can be found.

 Let $\lambda=(\lambda_1, \lambda_2, \ldots, \lambda_t) \vdash d$ with $\lambda_t>0$. Choose an integer $r \geq t$ and define a \emph{sequence of beta-numbers} by:
$$\beta_i=\lambda_i -i+r$$ for $1 \leq i \leq r$. For $r=s$ the $\beta$-numbers are just the first-column hook lengths in the Young diagram of $\lambda$. It is easily seen that $\lambda$ can be recovered uniquely from a corresponding sequence of beta numbers.

Now take an abacus with $p$ runners, labeled $0, 1, \ldots, p-1$ from left to right. Label the positions on the abacus so that row $s$ (counting down from the top)  runner $i$ is labeled $i+(s-1)p$. To represent $\lambda$ on the abacus with $r$ beads, simply place a bead at position $\beta_i$ for each $1 \leq i \leq r$.

Sliding all the beads in the abacus display as far up as possible, one obtains an abacus display for the $p$-core $\tilde{\lambda}$. Moving a bead up a single spot corresponds to removing a rim $p$-hook from the Young diagram of $\lambda$. Thus, if $w$ is the total number of such moves, then $\tilde{\lambda} \vdash d-pw$. The number $w$ is called the $p$-weight of $\lambda$. The following proposition collects some well-known facts, all found in the book \cite{JamesKerberbook}.

\begin{prop}
\label{prop:abacusproperties}
Let $\lambda \vdash d$  have $p$-weight $w$, $p$-core $\tilde{\lambda} \vdash d-pw$, and exactly $t$ parts. Let $\mu \vdash d$. Then:
\begin{itemize}
  \item[(a.)] The Specht modules $S^\lambda$ and $S^\mu$ lie in the same $p$-block if and only if $\tilde{\lambda}=\tilde{\mu}.$
  \item [(b.)]Suppose $B(\lambda)$ is the block containing $S^\lambda$. Then a defect group of $B(\lambda)$ is isomorphic to a Sylow $p$-subgroup of the symmetric group $\Sigma_{pw}$.
  \item[(c.)] Removing a removable node from $\lambda$ corresponds to sliding a bead in the abacus display of $\lambda$ from position $l$ to (unoccupied) position $l-1$. Adding an addable node corresponds to sliding a bead from position $l$ to (unoccupied) position $l+1$.

  \item[(d.)] If $\lambda$ is represented on an abacus with a multiple of $p$ beads then removable nodes which correspond to beads on runner $i$ have $p$-residue $i$.

  \item[(e.)] Let $\{A_1, A_2, \ldots, A_m\}$ denote the removable nodes of $\lambda$ and let $\lambda_A \vdash d-1$ denote $\lambda$ with removable node $A$ removed.  Then $\Res_{\Sigma_{d-1}}(S^\lambda)$ has a filtration by Specht modules where each $S^{\lambda_{A_i}}$ occurs once.
\end{itemize}
\end{prop}

\subsection{}Motivated by Proposition \ref{prop:abacusproperties}(d), we assume henceforth all our abacus displays have a multiple of $p$ beads. Given such an abacus display for $\lambda$, let $\lambda^{[i]}$ be the number of beads on runner $i$. Knowing these numbers is equivalent to knowing the $p$-core $\tilde{\lambda}$. For each runner we define a partition which captures how far each bead moves when sliding the beads of $\lambda$ all the way up. Define $\lambda_j^i$ as the number of empty spots above the $j$th bead from the top on runner $i$, i.e. the number of times this bead is moved up to obtain the $p$-core. Then $\lambda(i)=(\lambda_1^i, \lambda_2^i, \ldots )$ is a partition and the sequence of partitions $(\lambda(0), \lambda(1), \ldots, \lambda(p-1))$ is called the \emph{$p$-quotient} of $\lambda$. It should be clear that $\lambda$ is uniquely determined by its $p$-core and $p$-quotient, and that the parts of the $p$-quotient are partitions of integers which sum to $w$.

\section{The main theorem}
\label{sec: The main theorem}
\subsection{}
In this section we will prove Theorem \ref{thm: main}. Our first observation is that the $p \times p$ condition is easily described in terms of the abacus display:

\begin{prop}
\label{prop: pxpintermsofabacus} Let $\lambda \vdash d$ have $p$-quotient $(\lambda(0), \lambda(1), \ldots, \lambda(p-1))$.  Then $\lambda$ is \pbp if and only if $\tilde{\lambda}=\emptyset$ and $\lambda(0)=\lambda(1)= \cdots =\lambda(p-1)$.
\end{prop}
\begin{proof}
This is a straightforward exercise. Suppose $\lambda$ is \pbp and write $$\lambda=(p\tau_1^{pa_1}, p\tau_2^{pa_2}, \ldots, p\tau_r^{pa_r}).$$ Then $\lambda$ has $\sum_{i=1}^r pa_i$ parts and $p$-weight $\sum_{i=1}^r pa_i\tau_i$.  Representing $\lambda$ on an abacus with $\sum_{i=1}^r pa_i$ beads, one easily verifies that each $\lambda(i)$ in the $p$-quotient is just $(\tau_1^{a_1}, \tau_2^{a_2}, \ldots, \tau_r^{a_r})$. Conversely given a partition with constant $p$-quotient $\lambda(i)=\tau$ and empty $p$ core, one immediately observes that $\lambda$ is \pbp.
\end{proof}
\begin{rmk}
\label{rmk: abacusform}
The condition on $\lambda$ in Proposition \ref{prop: pxpintermsofabacus} is equivalent to each row of the abacus display being either empty or completely full. For example when $p=5$ the partition $\lambda=(20^{10},10^5,5^5)$ has abacus display shown in Figure \ref{fig: abacusdisplay}, where the corresponding $\tau$ is $(4,4,2,1)$.
\end{rmk}

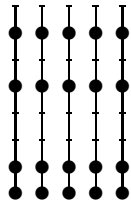
\begin{figure}[ht]
\begin{center}\setlength{\unitlength}{0.07in}
\begin{picture}(10,14)(0,0)

\put(2,0){\line(0,1){14}}
\put(4,0){\line(0,1){14}}
\put(6,0){\line(0,1){14}}
\put(8,0){\line(0,1){14}}
\put(10,0){\line(0,1){14}}
\put(1.8,0){\line(1,0){0.4}}\put(1.8,2){\line(1,0){0.4}}\put(1.8,4){\line(1,0){0.4}}
\put(1.8,6){\line(1,0){0.4}}\put(1.8,8){\line(1,0){0.4}}\put(1.8,10){\line(1,0){0.4}}
\put(1.8,12){\line(1,0){0.4}}\put(1.8,14){\line(1,0){0.4}}

\put(3.8,0){\line(1,0){0.4}}\put(3.8,2){\line(1,0){0.4}}\put(3.8,4){\line(1,0){0.4}}
\put(3.8,6){\line(1,0){0.4}}\put(3.8,8){\line(1,0){0.4}}\put(3.8,10){\line(1,0){0.4}}
\put(3.8,12){\line(1,0){0.4}}\put(3.8,14){\line(1,0){0.4}}

\put(5.8,0){\line(1,0){0.4}}\put(5.8,2){\line(1,0){0.4}}\put(5.8,4){\line(1,0){0.4}}
\put(5.8,6){\line(1,0){0.4}}\put(5.8,8){\line(1,0){0.4}}\put(5.8,10){\line(1,0){0.4}}
\put(5.8,12){\line(1,0){0.4}}\put(5.8,14){\line(1,0){0.4}}

\put(7.8,0){\line(1,0){0.4}}\put(7.8,2){\line(1,0){0.4}}\put(7.8,4){\line(1,0){0.4}}
\put(7.8,6){\line(1,0){0.4}}\put(7.8,8){\line(1,0){0.4}}\put(7.8,10){\line(1,0){0.4}}
\put(7.8,12){\line(1,0){0.4}}\put(7.8,14){\line(1,0){0.4}}

\put(9.8,0){\line(1,0){0.4}}\put(9.8,2){\line(1,0){0.4}}\put(9.8,4){\line(1,0){0.4}}
\put(9.8,6){\line(1,0){0.4}}\put(9.8,8){\line(1,0){0.4}}\put(9.8,10){\line(1,0){0.4}}
\put(9.8,12){\line(1,0){0.4}}\put(9.8,14){\line(1,0){0.4}}

\put(2,12){\circle*{1}}\put(4,12){\circle*{1}}\put(6,12){\circle*{1}}
\put(8,12){\circle*{1}}\put(10,12){\circle*{1}}

\put(2,8){\circle*{1}}\put(4,8){\circle*{1}}\put(6,8){\circle*{1}}
\put(8,8){\circle*{1}}\put(10,8){\circle*{1}}

\put(2,2){\circle*{1}}\put(4,2){\circle*{1}}\put(6,2){\circle*{1}}
\put(8,2){\circle*{1}}\put(10,2){\circle*{1}}

\put(2,0){\circle*{1}}\put(4,0){\circle*{1}}\put(6,0){\circle*{1}}
\put(8,0){\circle*{1}}\put(10,0){\circle*{1}}
\end{picture}

\end{center}

\caption{Abacus for $p=5, \lambda=(20^{10},15^5,10^{15})$}
\label{fig: abacusdisplay}
\end{figure}

\subsection{}
Suppose for the remainder of this section that $\lambda \vdash d$ has $p$-core $\tilde{\lambda} \vdash c$ and weight $w$, so $d=c+pw$. Suppose further that for every removable node $A$, the partition $\lambda_A$ which results from removing the node $A$ satisfies $w(\lambda_A) \leq w-2$. We will show that $\lambda$ has an abacus display as described in Remark \ref{rmk: abacusform}, and thus that $\lambda$ is \pbp. Recall the notation $\lambda^{[i]}$ for the number of beads on runner $i$, thus $\lambda(i) \vdash \lambda^{[i]}$.

\begin{lem}
\label{lem: removablerunner12etc} Suppose $\lambda$ has a removable node $A$  of residue $i> 0$ Then $\lambda^{[i-1]} \geq \lambda^{[i]}+1$.\
\end{lem}
\begin{proof}Removing a node corresponds to sliding a bead one spot to the left on the abacus. Suppose the bead corresponding to $A$ is in row $s$ of the abacus display, so there is no bead on runner $i-1$ in row $s$. Further, suppose runner $i-1$ has $y$ beads in the first $s-1$ rows and $\lambda^{[i-1]}-y$ beads in rows $>s$. Assume runner $i$ has $x$ beads in the first $s-1$ rows and $\lambda^{[i]}-x-1$ beads in rows $>s$. Runners $i-1$ and $i$ in the abacus display for $\lambda$ are illustrated in Figure \ref{fig: columni}.

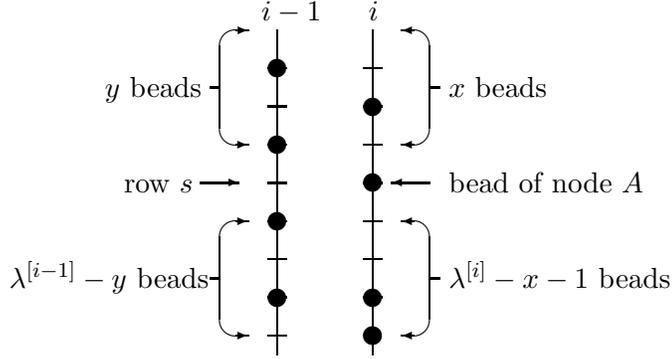
\begin{figure}[ht]

\begin{center}\setlength{\unitlength}{0.10in}
\begin{picture}(15,18)(0,0)

\put(5,-1){\line(0,1){17}}
\put(10,-1){\line(0,1){17}}
\put(4.5,2){\line(1,0){1}}\put(4.5,4){\line(1,0){1}}
\put(4.5,6){\line(1,0){1}}\put(4.5,8){\line(1,0){1}}
\put(4.5,10){\line(1,0){1}}\put(4.5,12){\line(1,0){1}}\put(4.5,14){\line(1,0){1}}
\put(4.5,0){\line(1,0){1}}
\put(9.5,2){\line(1,0){1}}\put(9.5,4){\line(1,0){1}}
\put(9.5,6){\line(1,0){1}}\put(9.5,8){\line(1,0){1}}
\put(9.5,10){\line(1,0){1}}\put(9.5,12){\line(1,0){1}}\put(9.5,14){\line(1,0){1}}

\put(10,8){\circle*{1}}
\put(10,12){\circle*{1}}
\put(10,2){\circle*{1}}
\put(10,0){\circle*{1}}
\put(5,6){\circle*{1}}
\put(5,10){\circle*{1}}
\put(5,2){\circle*{1}}
\put(5,14){\circle*{1}}
\put(10,0){\circle*{1}}

\put(4.2,16.5){$i-1$}
\put(9.8,16.5){$i$}
\put(-3,7.5){row $s$}
\put(1,8){\vector(1,0){2}}

\put(3,13){\oval(2,6)[l]}
\put(1.5,13){\line(1,0){0.5}}
\put(-4,12.5){$y$ beads}

\put(3,16){\vector(1,0){0.5}}
\put(3,10){\vector(1,0){0.5}}
\put(3,6){\vector(1,0){0.5}}
\put(3,0){\vector(1,0){0.5}}
\put(12,16){\vector(-1,0){0.5}}
\put(12,10){\vector(-1,0){0.5}}
\put(12,6){\vector(-1,0){0.5}}
\put(12,0){\vector(-1,0){0.5}}

\put(12,13){\oval(2,6)[r]}
\put(13,13){\line(1,0){0.5}}
\put(14,12.5){$x$ beads}

\put(14,7.5){bead of node $A$}
\put(13,8){\vector(-1,0){2}}
\put(3,3){\oval(2,6)[l]}
\put(1.5,3){\line(1,0){0.5}}
\put(-9,2.5){$\lambda^{[i-1]}-y$ beads}

\put(12,3){\oval(2,6)[r]}
\put(13,3){\line(1,0){0.5}}
\put(14,2.5){$\lambda^{[i]}-x-1$ beads}

\end{picture}

\end{center}

\caption{Removal of a node of residue $i>0$}
\label{fig: columni}
\end{figure}

Then the weight of bead $A$ changes from $s-x-1$ to $s-y-1$. In $\lambda_A$, the bottom $\lambda^{[i-1]}-y$ beads on runner $i-1$ are now weight one less and the bottom  $\lambda^{[i]}-x-1$ beads on runner $i$ are weight one more. Thus the change in weight between $\lambda$ and $\lambda_A$ is:

\begin{eqnarray}
\label{eq:changeinweight}
  w(\lambda_A)-w(\lambda) &=& s-y-1+\lambda^{[i]}-x-1-(s-x-1+\lambda^{[i-1]}-l)\\
   &=& \lambda^{[i]}-\lambda^{[i-1]}-1.  \nonumber
\end{eqnarray}

 Our assumption on $\lambda_A$ gives $\lambda^{[i]}-\lambda^{[i-1]}-1 \leq -2$, and thus $\lambda^{[i-1]} \geq \lambda^{[i]}+1$.
\end{proof}

\begin{prop}
\label{prop: nonincreonpcore}The number of beads on each runner in the abacus display of $\lambda$ is nonincreasing from left to right, i.e. $\lambda^{[i-1]} \geq \lambda^{[i]}.$
\end{prop}
\begin{proof} If $\lambda^{[i-1]} < \lambda^{[i]}$ then the extra beads on runner $i$ ensure $\lambda$ has a removable node of residue $i$, which then contradicts Lemma \ref{lem: removablerunner12etc}.
\end{proof}

Next we consider the situation of a removable node of residue zero.

\begin{lem}
\label{lem: removableresidue0}Suppose $\lambda$ has a removable node of residue zero. Then $\lambda^{[0]}=\lambda^{[p-1]}.$
\end{lem}
\begin{proof}Suppose $\lambda$ has a removable node of residue zero, so the abacus display has a bead in row $s+1$ of runner zero, and no bead in row $s$ of runner $p-1$. We let the reader draw the corresponding picture to Figure \ref{fig: columni}, and observe that
\begin{equation}
\label{eq: residuezero}
w(\lambda_A)-w(\lambda)=\lambda^{[0]}-\lambda^{[p-1]}-2.
\end{equation}
 By Prop. \ref{prop: nonincreonpcore} plus our assumption that $w(\lambda_A)-w(\lambda) \leq -2$, we obtain the result.
\end{proof}
If any of the inequalities in Proposition \ref{prop: nonincreonpcore} were strict, then we would have $\lambda^{[0]}>\lambda^{[p-1]}$, which would contradict Lemma \ref{lem: removableresidue0}. Thus we have proven that:

\begin{equation}
\label{eq: all lambaeq}
\lambda^{[0]}=\lambda^{[1]} \cdots =\lambda^{[p-1]}.
\end{equation}
Comparing equations \eqref{eq:changeinweight} and \eqref{eq: all lambaeq}, we discover that all removable nodes of $\lambda$ have weight zero, since removable nodes of any other residue decrease the weight of the partition by only one. So the abacus display of $\lambda$ has the same number of beads on each runner, and all the removable nodes have residue zero. This forces the abacus display to be as described in Remark \ref{rmk: abacusform}, and completes the proof of Theorem \ref{thm: main}. Notice that \eqref{eq: residuezero} and \eqref{eq: all lambaeq} implies $w(\lambda_A)$ is always exactly $w-2$, although our assumption was only that it is $\leq w-2$.

\section{Complexity}
\subsection{}
Finally we prove Corollary \ref{cor: onewayofconjecture}, which gives one direction of Conjecture \ref{conj: Vigre}. A module has complexity $c$ if the dimensions in a minimal projective resolution of that module are bounded by a polynomial of degree $c-1$. The complexity of the trivial module is known to be the maximal rank of an elementary abelian subgroup, which for $\Sigma_{pd}$ is just $d$. In general the complexity of a module for the symmetric group in a block of weight $w$ is at most $w$. Good  for this theory are \cite{Bensonrepteoryvolume2} or \cite{EvensCohobook}.

So suppose $\lambda \vdash p^2d$ is \pbp and so has weight $w=pd$. By Proposition \ref{prop:abacusproperties}(e) and Theorem \ref{thm: main}, all summands of $\Res_{\Sigma_{p^2d-1}}S^\lambda$ lie in blocks of weight $w-2$. By \cite[4.2.1d]{HNsupportvariety}, $\Res_{\Sigma_{p^2d-1}}S^\lambda$ has complexity at most $w-2$. An elementary argument using support varieties (see \cite{LimVarietySpechtpreprint} for example) implies that $S^\lambda$ cannot have complexity $w$. If it did, then for $E$ a maximal elementary abelian subgroup of rank $w$, the support variety $V_E(S^\lambda)$ would have dimension $w$, which would force the complexity of $\Res_{\Sigma_{p^2d-1}}S^\lambda$ to be $w-1.$

\begin{rmk}For the case $\lambda=p^p$, dimension considerations immediately give the complexity of $S^{(p^{p-1},p-1)}$, which lets Lim deduce \cite{LimVarietySpechtpreprint} that the complexity of $S^{(p^p)}$ is exactly $p-1$. This part of the argument does not appear to generalize, i.e. in general it large powers of $p$ can divide the dimension of $S^{\lambda_A}$ when $\lambda$ is \pbp.
\end{rmk}

\section{Problems}
There are several obvious problems left unsolved.

\begin{prob} Resolve the other direction of Conjecture \ref{conj: Vigre}.
\end{prob}
\begin{prob}Suppose $\lambda$ is \pbp of weight $w$. Is the complexity of $S^\lambda$ equal to $w-1$, or can it drop further?
\end{prob}
\begin{prob}One can generalize the definition of \pbp. For example the first obvious generalization would be to require $\lambda$ be \pbp and each $\lambda(i)$ be \pbp. (And recursively for higher generalizations.) Can one say anything interesting about these situations? Perhaps the complexity drops by even more in this case?
\end{prob}

\bibliography{references0808}
\bibliographystyle{plain}
\end{document}